\documentclass[Svgc_modified]{Svgc_modified}
\usepackage{amsmath}
\usepackage{graphicx}
\usepackage{enumitem}
\usepackage{subfigure}
\usepackage{stmaryrd}
\usepackage{tikz,xcolor}
\usepackage[utf8]{inputenc}
\usepackage[T1]{fontenc}
\def\1{$\ }
\def\2{\ $}

\def\cE{{\cal E}}
\def\cH{{\cal H}}

\def\hqed{\hfill\framebox(6,6){\ }}

\def\ul{\underline}

\def\eiE{\cE^{EI}}

\begin{document}
%
\title{Nearly all cacti are edge intersection hypergraphs of  3-uniform hypergraphs}
\author{Martin Sonntag\inst{1} \and Hanns-Martin Teichert\inst{2}}

%
%
\institute{Faculty of Mathematics and Computer Science, Technische Universit\"{a}t Bergakademie Freiberg, Pr\"{u}ferstra\ss e 1, 09596 Freiberg, Germany
\and Institute of Mathematics, University of L\"{u}beck, Ratzeburger Allee 160, 23562 L\"{u}beck, Germany
}
\maketitle
\begin{abstract}
If $\cH=(V,\cE)$ is a hypergraph, its {\it edge intersection hypergraph} $EI(\cH)=(V,\eiE)$ has the edge set $\eiE=\{e_1 \cap e_2 \ |\ e_1, e_2 \in \cE \ \wedge \ e_1 \neq e_2  \ \wedge  \ |e_1 \cap e_2 |\geq2\}$. Using the so-called {\em clique-fusion}, we show that nearly all cacti
are edge intersection hypergraphs of 3-uniform hypergraphs. In the proof we make use of known characterizations of the trees and the cycles which are edge intersection hypergraphs of 3-uniform hypergraphs  (cf. \cite{ST2019_1}).
\end{abstract}
\begin{keyword}
Edge intersection hypergraph, cactus
\end{keyword}
\small{\bf  Mathematics Subject Classification 2010:} 05C65
%
\section{Introduction and basic definitions}
All hypergraphs $\cH = (V(\cH), \cE(\cH))$ and  graphs $G=(V(G), E(G))$  considered in the following may have isolated vertices but no multiple edges or loops.

A hypergraph $\cH = (V, \cE)$ is {\em $k$-uniform} if all hyperedges $e \in \cE$ have the cardinality $k$.
Trivially, any 2-uniform hypergraph $\cH$ is a graph.
The {\em degree} $d(v)$ (or $d_{\cH}(v)$) of a vertex $v \in V$ is the number of hyperedges $e \in \cE$ being incident to the vertex $v$. In standard terminology we follow Berge \cite{GST5}.

If $\cH=(V,\cE)$ is a hypergraph, its {\it edge intersection hypergraph} $EI(\cH)=(V,\eiE)$ has the edge set $\eiE=\{e_1 \cap e_2 \ |\ e_1, e_2 \in \cE \ \wedge \ e_1 \neq e_2  \ \wedge \ |e_1 \cap e_2 |\geq2\}$.

For an application, as well as for more details on structural properties of edge intersection hypergraphs, see \cite{ST2019_1}.

Obviously, for certain hypergraphs $\cH$ the edge intersection hypergraph $EI(\cH)$  can be 2-uniform; in this case $EI(\cH)$ is a simple, undirected graph $G$. But in contrast to the notions {\em intersection graph} or  {\em edge intersection graph} (cf. \cite{Bie}, \cite{Cam}, \cite{Gol}, \cite{Nai} and \cite{Sku}),
   $G=EI(\cH)$ and $\cH$ have one and the same vertex set $V(G) = V(\cH)$.
Therefore we consistently use our notation "edge intersection hypergraph" also when this hypergraph is 2-uniform.

First of all, in Section 2 we introduce a powerful tool for the construction of edge intersection hypergraphs,  the so-called {\em clique-fusion}.

Then we restrict our investigations to cacti.
A simple, connected graph $G=(V,E)$ is referred to as a {\em cactus} if and only if every edge $e \in E$ is contained in at most one cycle of $G$. The {\em circumference} $ci(G)$ of a graph $G$ is the length of a longest cycle in $G$.

In Section 3 we describe a special decomposition of cacti into trees and cycles and make use of
two important results from \cite{ST2019_1}.
The first one  includes that all but seven exceptional trees are edge intersection hypergraphs of 3-uniform hypergraphs. The exceptional trees have at most 6 vertices. Besides the notation $P_n$ for the path with $n$ vertices, in the subsequent Theorem 1 we find two additional notations, namely $T7$ and $T12$ of trees with 5 and 6 vertices, respectively. These notations come from \cite{RW}, they denote the graphs \\
$T7 = ( V = \{v_1,v_2, \ldots, v_5 \}, E= \{\{v_1,v_2\}, \{v_2,v_3\}, \{v_3,v_4\}, \{v_3,v_5\}\})$ and \\
$T12 = ( V= \{v_1,v_2, \ldots, v_6 \}, E= \{\{v_1,v_2\}, \{v_2,v_3\}, \{v_3,v_4\}, \{v_2,v_5\}, \{v_5,v_6\}\})$ (see Fig.1).

\providecolor{black}{rgb}{0,0,0}
	\begin{figure}[h]
	\centering
	\begin{tikzpicture}


		\path[solid,line cap=round,draw=black, line width=0.02586206896551724cm] (1.206896551724138,2.4224137931034484) --  (1.206896551724138,3.4051724137931036);

		\path[solid,line cap=round,draw=black, line width=0.02586206896551724cm] (1.206896551724138,1.4396551724137931) --  (1.206896551724138,2.4224137931034484);

		\path[solid,line cap=round,draw=black, line width=0.02586206896551724cm] (2.189655172413793,0.45689655172413796) --  (1.206896551724138,1.4396551724137931);

		\path[solid,line cap=round,draw=black, line width=0.02586206896551724cm] (0.22413793103448276,0.45689655172413796) --  (1.206896551724138,1.4396551724137931);

		\path[solid,line cap=round,draw=black, line width=0.02586206896551724cm] (6.448275862068965,2.4224137931034484) --  (6.448275862068965,3.4051724137931036);

		\path[solid,line cap=round,draw=black, line width=0.02586206896551724cm] (7.431034482758621,1.4396551724137931) --  (6.448275862068965,2.4224137931034484);

		\path[solid,line cap=round,draw=black, line width=0.02586206896551724cm] (7.431034482758621,1.4396551724137931) --  (7.431034482758621,0.45689655172413796);

		\path[solid,line cap=round,draw=black, line width=0.02586206896551724cm] (6.448275862068965,2.4224137931034484) --  (5.4655172413793105,1.4396551724137931);

		\path[solid,line cap=round,draw=black, line width=0.02586206896551724cm] (5.4655172413793105,1.4396551724137931) --  (5.4655172413793105,0.45689655172413796);

		\fill[fill=black] (1.206896551724138,3.4051724137931036) circle (0.06465517241379311);
		\pgftext[x=1.5517241379310345cm,y=3.4051724137931036cm]{$v_{1}$}

		\fill[fill=black] (1.206896551724138,2.4224137931034484) circle (0.06465517241379311);
		\pgftext[x=1.5517241379310345cm,y=2.4224137931034484cm]{$v_{2}$}

		\fill[fill=black] (1.206896551724138,1.4396551724137931) circle (0.06465517241379311);
		\pgftext[x=1.5517241379310345cm,y=1.4396551724137931cm]{$v_{3}$}

		\fill[fill=black] (2.189655172413793,0.45689655172413796) circle (0.06465517241379311);
		\pgftext[x=2.189655172413793cm,y=0.11206896551724138cm]{$v_{4}$}

		\fill[fill=black] (0.22413793103448276,0.45689655172413796) circle (0.06465517241379311);
		\pgftext[x=0.2241379310344827cm,y=0.11206896551724138cm]{$v_{5}$}

		\fill[fill=black] (6.448275862068965,3.4051724137931036) circle (0.06465517241379311);
		\pgftext[x=6.793103448275862cm,y=3.4051724137931036cm]{$v_{1}$}

		\fill[fill=black] (6.448275862068965,2.4224137931034484) circle (0.06465517241379311);
		\pgftext[x=6.793103448275862cm,y=2.4224137931034484cm]{$v_{2}$}

		\fill[fill=black] (5.4655172413793105,1.4396551724137931) circle (0.06465517241379311);
		\pgftext[x=5.120689655172414cm,y=1.4396551724137931cm]{$v_{5}$}

		\fill[fill=black] (7.431034482758621,1.4396551724137931) circle (0.06465517241379311);
		\pgftext[x=7.775862068965518cm,y=1.4396551724137931cm]{$v_{3}$}

		\fill[fill=black] (7.431034482758621,0.45689655172413796) circle (0.06465517241379311);
		\pgftext[x=7.431034482758621cm,y=0.11206896551724138cm]{$v_{4}$}

		\fill[fill=black] (5.4655172413793105,0.45689655172413796) circle (0.06465517241379311);
		\pgftext[x=5.4655172413793105cm,y=0.11206896551724138cm]{$v_{6}$}
	\end{tikzpicture}
		\caption{ \hspace{35mm} T7 \hspace{44mm}  T12}
	\end{figure}

\begin{theorem}(Theorem 6 in \cite{ST2019_1}).
All trees but $P_2, P_3, P_4, P_5, P_6, T7$ and $T12$ are edge intersection hypergraphs of a 3-uniform hypergraph $\cH$.
\end{theorem}

Note that Theorem 1 characterizes the cycle-free cacti being edge intersection hypergraphs of 3-uniform hypergraphs.
The following characterization of the cycles which are edge intersection hypergraphs of 3-uniform hypergraphs is very simple to verify.

\begin{theorem}(Corollary 1 in \cite{ST2019_1}).
For $n \ge 5$ the cycle $C_n$ is an edge intersection hypergraph of a 3-uniform hypergraph.
\end{theorem}

Using the clique-fusion, we prove that cacti having either a circumference of at least 5 or containing (in the decomposition mentioned above) a tree $T \notin  \{ P_1,P_2, P_3, P_4, P_5, P_6, T7, T12 \}$ are edge intersection hypergraphs of 3-uniform hypergraphs.

\smallskip
At the end of the introduction, let us mention a  tool from \cite{ST2019_1}, which is useful for the investigation of small examples. For this end let $G=(V,E)$ and $\cH =(V, \cE)$ be a graph and a hypergraph, respectively, having one and the same vertex set $V$.
The verification of $\cE(EI(\cH)) = E(G)$ can be done  by hand
 or by computer, e.g. using the computer algebra system MATHEMATICA
 \cite{WMath} with the function
\\[0.8ex]
\indent
$EEI[eh\_] :=
 Complement[
  Select[Union[Flatten[Outer[Intersection, eh, eh, 1], 1]],$
\vspace{-1ex}
\begin{flushright} $   Length[\#] > 1 \&], eh],$ \end{flushright}

\smallskip
\noindent
where the argument $eh$ has to be the list of the hyperedges of $\cH$ in the form $\{\{a,b,c \}, \ldots, \{x,y,z \} \}$. Then $EEI[eh]$ provides the list of the hyperedges of $EI(\cH)$.

\section{The clique-fusion and  trees and cycles}


Let $r \ge 2$ and $G_1 = (V_1, E_1), G_2 = (V_2, E_2), \ldots, G_r = (V_r, E_r)$ be graphs.
Moreover, let $k \ge 1$, $V' = \{ v_1, \ldots, v_k \} := \{ v \, | \, \exists i, j \in \{1, \ldots, r \}: \, i \neq j \, \wedge \, v \in  V_i \cap V_ j \}$  and \\
 $E' := \{ \{v, v' \} \, | \, v, v' \in V'\, \wedge \, \exists i \in  \{1, \ldots r \}: \, \{v, v' \} \in E_i \}$.

\smallskip
Incidentally, if for each $i \in \{1, \ldots, r \}$ the graph $G_i$ is connected and $V_i \cap V' \neq \emptyset$, then
the union $G_1 \cup \ldots \cup G_r = ( V_1 \cup \ldots \cup V_r , E_1 \cup \ldots \cup E_r )$ is connected, too.

\smallskip
Consider the case that $E' = \{ \{ v_i, v_j \} \, | \, 1 \le i < j \le k \}$, i.e. the subgraph $ \langle V' \rangle_{G_1 \cup \ldots \cup G_r}$ induced by the vertices of $V'$ in $G_1 \cup \ldots \cup G_r$ is a $k$-clique.
Then we refer to the union $G_1 \cup \ldots \cup G_r$ as the {\em clique-fusion} or {\em $k$-fusion}
of the graphs $G_1 \cup \ldots \cup G_r$ and write $G_1 \oplus \ldots \oplus G_r := G_1 \cup \ldots \cup G_r$.

\smallskip
For an example, consider three graphs $G_1 = (V_1, E_1)$, $G_2 = (V_2, E_2)$ and $G_3 = (V_3, E_3)$, where $\{x, y \} \in E_1$, $\{y, z \} \in E_2$ and $\{x, z \} \in E_3$ are edges as well as
$V_1 \cap V_2 = \{y \}$, $V_2 \cap V_3 = \{z \}$ and $V_1 \cap V_3 = \{x \}$ hold. Then $V' = \{ x,y,z \}$ induces a 3-clique $ \langle \{ x,y,z \} \rangle_{G_1 \cup G_2 \cup G_3}$ in $ G_1 \cup G_2 \cup G_3$ and $ G_1 \oplus G_2 \oplus G_3 =  G_1 \cup G_2 \cup G_3 $ is the $3$-fusion of the graphs $G_1, G_2$ and $G_3$ (see Fig.2).

\providecolor{black}{rgb}{0,0,0}
	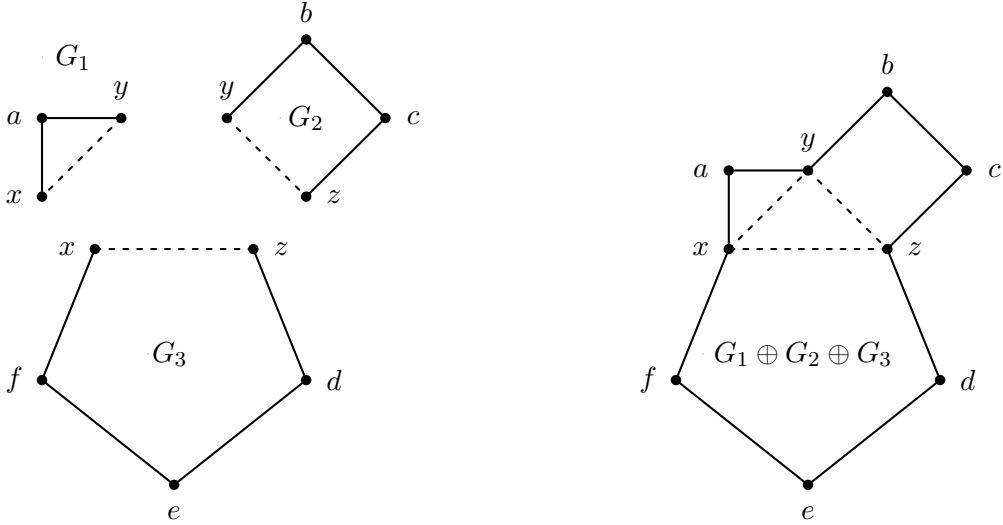
\begin{figure}[h]
	\centering
	\begin{tikzpicture}


		\path[solid,line cap=round,draw=black, line width=0.02742616033755274cm] (10.49507735583685,4.653305203938115) --  (9.452883263009845,4.653305203938115);

		\path[solid,line cap=round,draw=black, line width=0.02742616033755274cm] (9.452883263009845,4.653305203938115) --  (9.452883263009845,3.611111111111111);

		\path[ dash pattern=on 0.0639943741209564cm off 0.11884669479606189cm,line cap=round,draw=black, line width=0.02742616033755274cm] (9.452883263009845,3.611111111111111) --  (10.49507735583685,4.653305203938115);

		\path[solid,line cap=round,draw=black, line width=0.02742616033755274cm] (11.537271448663853,5.695499296765119) --  (10.49507735583685,4.653305203938115);

		\path[solid,line cap=round,draw=black, line width=0.02742616033755274cm] (12.579465541490858,4.653305203938115) --  (11.537271448663853,5.695499296765119);

		\path[solid,line cap=round,draw=black, line width=0.02742616033755274cm] (11.537271448663853,3.611111111111111) --  (12.579465541490858,4.653305203938115);

		\path[ dash pattern=on 0.0639943741209564cm off 0.11884669479606189cm,line cap=round,draw=black, line width=0.02742616033755274cm] (10.49507735583685,4.653305203938115) --  (11.537271448663853,3.611111111111111);

		\path[ dash pattern=on 0.0639943741209564cm off 0.11884669479606189cm,line cap=round,draw=black, line width=0.02742616033755274cm] (11.537271448663853,3.611111111111111) --  (9.452883263009845,3.611111111111111);

		\path[solid,line cap=round,draw=black, line width=0.02742616033755274cm] (11.537271448663853,3.611111111111111) --  (12.232067510548523,1.8741209563994374);

		\path[solid,line cap=round,draw=black, line width=0.02742616033755274cm] (12.232067510548523,1.8741209563994374) --  (10.49507735583685,0.48452883263009844);

		\path[solid,line cap=round,draw=black, line width=0.02742616033755274cm] (8.758087201125175,1.8741209563994374) --  (10.49507735583685,0.48452883263009844);

		\path[solid,line cap=round,draw=black, line width=0.02742616033755274cm] (8.758087201125175,1.8741209563994374) --  (9.452883263009845,3.611111111111111);

		\path[solid,line cap=round,draw=black, line width=0.02742616033755274cm] (1.4627285513361463,5.348101265822785) --  (0.42053445850914206,5.348101265822785);

		\path[solid,line cap=round,draw=black, line width=0.02742616033755274cm] (0.42053445850914206,4.305907172995781) --  (0.42053445850914206,5.348101265822785);

		\path[solid,line cap=round,draw=black, line width=0.02742616033755274cm] (4.936708860759493,5.348101265822785) --  (3.894514767932489,6.390295358649789);

		\path[ dash pattern=on 0.0639943741209564cm off 0.11884669479606189cm,line cap=round,draw=black, line width=0.02742616033755274cm] (2.852320675105485,5.348101265822785) --  (3.894514767932489,4.305907172995781);

		\path[solid,line cap=round,draw=black, line width=0.02742616033755274cm] (2.852320675105485,5.348101265822785) --  (3.894514767932489,6.390295358649789);

		\path[solid,line cap=round,draw=black, line width=0.02742616033755274cm] (3.894514767932489,4.305907172995781) --  (4.936708860759493,5.348101265822785);

		\path[solid,line cap=round,draw=black, line width=0.02742616033755274cm] (3.19971870604782,3.611111111111111) --  (3.894514767932489,1.8741209563994374);

		\path[ dash pattern=on 0.0639943741209564cm off 0.11884669479606189cm,line cap=round,draw=black, line width=0.02742616033755274cm] (3.19971870604782,3.611111111111111) --  (1.1153305203938115,3.611111111111111);

		\path[solid,line cap=round,draw=black, line width=0.02742616033755274cm] (2.157524613220816,0.48452883263009844) --  (3.894514767932489,1.8741209563994374);

		\path[solid,line cap=round,draw=black, line width=0.02742616033755274cm] (2.157524613220816,0.48452883263009844) --  (0.42053445850914206,1.8741209563994374);

		\path[solid,line cap=round,draw=black, line width=0.02742616033755274cm] (0.42053445850914206,1.8741209563994374) --  (1.1153305203938115,3.611111111111111);

		\path[ dash pattern=on 0.0639943741209564cm off 0.11884669479606189cm,line cap=round,draw=black, line width=0.02742616033755274cm] (1.4627285513361463,5.348101265822785) --  (0.42053445850914206,4.305907172995781);

		\fill[fill=black] (9.452883263009845,4.653305203938115) circle (0.06856540084388185);
		\pgftext[x=9.087201125175808cm,y=4.653305203938115cm]{$a$}

		\fill[fill=black] (10.49507735583685,4.653305203938115) circle (0.06856540084388185);
		\pgftext[x=10.49507735583685cm,y=5.018987341772151cm]{$y$}

		\fill[fill=black] (9.452883263009845,3.611111111111111) circle (0.06856540084388185);
		\pgftext[x=9.087201125175808cm,y=3.611111111111111cm]{$x$}

		\fill[fill=black] (11.537271448663853,3.611111111111111) circle (0.06856540084388185);
		\pgftext[x=11.90295358649789cm,y=3.611111111111111cm]{$z$}

		\fill[fill=black] (11.537271448663853,5.695499296765119) circle (0.06856540084388185);
		\pgftext[x=11.537271448663853cm,y=6.061181434599155cm]{$b$}

		\fill[fill=black] (12.579465541490858,4.653305203938115) circle (0.06856540084388185);
		\pgftext[x=12.945147679324895cm,y=4.653305203938115cm]{$c$}

		\fill[fill=black] (8.758087201125175,1.8741209563994374) circle (0.06856540084388185);
		\pgftext[x=8.392405063291138cm,y=1.8741209563994374cm]{$f$}

		\fill[fill=black] (12.232067510548523,1.8741209563994374) circle (0.06856540084388185);
		\pgftext[x=12.59774964838256cm,y=1.8741209563994374cm]{$d$}

		\fill[fill=black] (10.49507735583685,0.48452883263009844) circle (0.06856540084388185);
		\pgftext[x=10.49507735583685cm,y=0.11884669479606186cm]{$e$}

		\fill[fill=black] (3.894514767932489,1.8741209563994374) circle (0.06856540084388185);
		\pgftext[x=4.260196905766525cm,y=1.8741209563994374cm]{$d$}

		\fill[fill=black] (0.42053445850914206,1.8741209563994374) circle (0.06856540084388185);
		\pgftext[x=0.05485232067510548cm,y=1.8741209563994374cm]{$f$}

		\fill[fill=black] (2.157524613220816,0.48452883263009844) circle (0.06856540084388185);
		\pgftext[x=2.157524613220816cm,y=0.11884669479606186cm]{$e$}

		\fill[fill=black] (3.19971870604782,3.611111111111111) circle (0.06856540084388185);
		\pgftext[x=3.5654008438818567cm,y=3.611111111111111cm]{$z$}

		\fill[fill=black] (1.1153305203938115,3.611111111111111) circle (0.06856540084388185);
		\pgftext[x=0.749648382559775cm,y=3.611111111111111cm]{$x$}

		\fill[fill=black] (0.42053445850914206,4.305907172995781) circle (0.06856540084388185);
		\pgftext[x=0.05485232067510548cm,y=4.305907172995781cm]{$x$}

		\fill[fill=black] (0.42053445850914206,5.348101265822785) circle (0.06856540084388185);
		\pgftext[x=0.05485232067510548cm,y=5.348101265822785cm]{$a$}

		\fill[fill=black] (1.4627285513361463,5.348101265822785) circle (0.06856540084388185);
		\pgftext[x=1.4627285513361463cm,y=5.713783403656821cm]{$y$}

		\fill[fill=black] (3.894514767932489,4.305907172995781) circle (0.06856540084388185);
		\pgftext[x=4.260196905766525cm,y=4.305907172995781cm]{$z$}

		\fill[fill=black] (2.852320675105485,5.348101265822785) circle (0.06856540084388185);
		\pgftext[x=2.852320675105485cm,y=5.713783403656821cm]{$y$}

		\fill[fill=black] (3.894514767932489,6.390295358649789) circle (0.06856540084388185);
		\pgftext[x=3.894514767932489cm,y=6.755977496483825cm]{$b$}

		\fill[fill=black] (4.936708860759493,5.348101265822785) circle (0.06856540084388185);
		\pgftext[x=5.30239099859353cm,y=5.348101265822785cm]{$c$}

		\fill[fill=black] (0.42053445850914206,6.152601969057665) circle (0.0);
		\pgftext[x=0.8319268635724332cm,y=6.152601969057665cm]{$G_1$}

		\fill[fill=black] (3.5471167369901546,5.348101265822785) circle (0.0);
		\pgftext[x=3.894514767932489cm,y=5.348101265822785cm]{$G_2$}

		\fill[fill=black] (1.810126582278481,2.221518987341772) circle (0.0);
		\pgftext[x=2.10267229254571cm,y=2.221518987341772cm]{$G_3$}

		\fill[fill=black] (9.10548523206751,2.221518987341772) circle (0.0);
		\pgftext[x=10.431082981715893cm,y=2.221518987341772cm]{$G_1 \oplus G_2 \oplus G_3$}
	\end{tikzpicture}

		\caption{Three graphs and their  clique-fusion}
	\end{figure}

Note that we obtain the same $3$-fusion taking the modified graphs $G'_1 = (V_1 \cup \{ z \}, E_1 \cup \{ \{x, z \}, \{y, z \} \})$ and $G'_3 = (V_3, E_3 \setminus \{ \{x, z \} \})$ instead of $G_1$ and $G_3$, i.e. we have $ G_1 \oplus G_2 \oplus G_3 = G'_1 \oplus G_2 \oplus G'_3$.

Using the above notations  we have a look at a special situation. \\[1ex]
{\bf Special Case 1.}
$\quad \forall i,j \in \{ 1, \ldots, r \}: \; V' = V_i \cap V_j \; \wedge \; \langle V' \rangle_{G_i} \mbox{ is a $k$-clique}.$

\medskip
In this case, all graphs $G_i, G_j$ $(i \neq j)$ have all the vertices $v_1, \ldots, v_k$ (and only these vertices) in common. Additionally, in each $G_i$ $( i \in \{ 1, \dots, k \})$ (as well as in $ \langle V' \rangle_{G_1 \cup \ldots \cup G_r}$) the vertices  $v_1, \ldots, v_k$ induce one and the same $k$-clique.

Let us mention two further special cases; the first one corresponds to $k=1$ and the second one  to $r=2$, respectively.\\[1ex]
{\bf Special Case 2.}
$\quad \forall i,j \in \{ 1, \ldots, r \}: \; V' = \{ v \} = V_i \cap V_j ,$
where $v$ is a uniquely determined vertex. \\[1ex]
{\bf Special Case 3.}
$\quad r=2$, i.e. we consider the clique-fusion $ G_1 \oplus G_2$ of two graphs.

\medskip
Investigating cacti, we only need Special Case 2 and Special Case 3 in combination, i.e. we have $k=1$ as well as $r=2$. Only for proof-technical reasons, in  very few exceptions we use the 2-fusion (see $G_1 \oplus K_{1,3}$ in the part (b) of the proof of Theorem 4).

\medskip

\begin{remark}
The clique-fusion can be easily generalized to pairwise vertex-disjoint graphs $G_1 = (V_1, E_1), G_2 = (V_2, E_2), \ldots, G_r = (V_r, E_r)$ by identifying certain vertices $v^i \in V_i$ and $v^j \in V_j \,$: $v^i \equiv v^j$, where  $1 \le i < j \le r$. Obviously, also the edges $\{ v_1^i, v_2^i \} \in E_i$ and  $\{ v_1^j, v_2^j \} \in E_j$ have to be identified ($\{ v_1^i, v_2^i \} \equiv \{ v_1^j, v_2^j \}$), if the corresponding vertices have been identified ($v_1^i \equiv v_1^j$ and $v_2^i \equiv v_2^j$).
\end{remark}

In the following, we will only make use of the clique-fusion in its original form, not in the generalized sense described in Remark 1.

Now we prove that the clique-fusion of graphs
which are edge intersection hypergraphs of 3-uniform hypergraphs
is an edge intersection hypergraph of a 3-uniform hypergraph,
too.

\begin{theorem}
Let $G = (V, E)$ be the clique-fusion $G_1 \oplus \ldots \oplus G_r$ of graphs $G_1 = (V_1, E_1),  \ldots, G_r = (V_r, E_r)$, where $G_1 = EI(\cH_1), \ldots, G_r = EI(\cH_r)$ are edge intersection hypergraphs of the 3-uniform hypergraphs $\cH_1 = (V_1, \cE_1), \ldots, \cH_r = ( V_r, \cE_r)$.

Then $\cH = \cH_1 \cup \ldots \cup \cH_r$ is 3-uniform and $G = EI(\cH)
 $.
\end{theorem}

\begin{proof}
The 3-uniformity of $\cH = (V, \cE)$ is trivial because of $V = V_1 \cup \ldots \cup V_r$ and $\cE = \cE_1 \cup \ldots \cup \cE_r$. Owing to $G = G_1 \oplus \ldots \oplus G_r = G_1 \cup \ldots \cup G_r = EI(\cH_1) \cup \ldots \cup EI(\cH_r)$ in addition to $V = V_1 \cup \ldots \cup V_r $ we have $E = E_1 \cup \ldots \cup E_r = \cE(EI(\cH_1)) \cup \ldots \cup \cE(EI(\cH_r))$ and it suffices to show $ \cE(EI(\cH)) =
\cE(EI(\cH_1)) \cup \ldots \cup \cE(EI(\cH_r))$.
(Note that, in general, neither $V_1, \ldots, V_r$ nor $E_1 = \cE(EI(\cH_1)), \ldots, E_r = \cE(EI(\cH_r))$ are pairwise distinct.)

\bigskip
\ul{Part 1: $\cE(EI(\cH))  \supseteq \cE(EI(\cH_1)) \cup \ldots \cup \cE(EI(\cH_r))$}. \\[0.5ex]
We consider an arbitrarily chosen  edge $\{ x, y \} \in \cE(EI(\cH_i))$, where $ i \in \{ 1, \ldots, r \}$. Then there are hyperedges $e_i, e_i' \in \cE(\cH_i) \subseteq \cE(\cH)$ with $ e_i \cap e'_i = \{ x, y \}$. This implies $\{ x, y \} \in \cE(EI(\cH))$.

\bigskip
\ul{Part 2: $\cE(EI(\cH))  \subseteq \cE(EI(\cH_1)) \cup \ldots \cup \cE(EI(\cH_r))$}. \\[0.5ex]
Let $e, e' \in \cE(\cH)$ with $ e \cap e' \in \cE(EI(\cH))$. The 3-uniformity of $\cH$ includes $| e \cap e'| = 2$.

Assume, there exists an $i \in \{1, \ldots, r \}$ such that $e, e' \in \cE(\cH_i)$. Then $e \cap e' \in \cE(EI(\cH_i)) \subseteq \cE(EI(\cH_1)) \cup \ldots \cup \cE(EI(\cH_r))$.

Otherwise, for all $i, j \in \{ 1, \ldots, r \}$ from $e \in \cE(\cH_i)$ and  $e' \in \cE(\cH_j)$ we get $i \neq j$. In this case we have $e \cap e' = \{ x, y \}$, with $\{ x, y \}  \subseteq V_i \cap V_j$ and $x \neq y$. Since $G$ is the clique-fusion $G_1 \oplus \ldots \oplus G_r$, the vertices $x$ and $y$ have to be adjacent in $G$ and, therefore, $e \cap e' = \{ x, y \} \in E(G) = E = E_1 \cup \ldots \cup E_r = \cE(EI(\cH_1)) \cup \ldots \cup \cE(EI(\cH_r))$. \\[0.5ex]
(Note that because of  $e \cap e' = \{ x, y \} \in  E_1 \cup \ldots \cup E_r$ there exist an $l \in \{1, \ldots, r \}$ and hyperedges $e_l, e_l' \in \cE(\cH_l)$ with $e_l \cap e_l' = \{ x, y \}$. Therefore, we even get $e \cap e' = e_l \cap e_l' \in \cE(EI(\cH_l)) = E_l$.)\hqed

\end{proof}

Let the graph $G_1$ be the edge intersection hypergraph of a 3-uniform hypergraph $\cH_1$. Besides the clique-fusion of  $G_1$ and a graph $G_2$ being the edge intersection hypergraph of a second 3-uniform hypergraph $\cH_2$ (see Theorem 3), we will need also the clique-fusion of $G_1$ with arbitrary trees (cf. Theorem 4) and cycles (cf. Theorem 5), respectively.

\begin{theorem}
Let the graph $G_1 = (V_1, E_1)$ be the edge intersection hypergraph of a 3-uniform hypergraph $\cH_1=(V_1, \cE_1)$, $v \in V_1$ with $d_{G_1}(v) \ge 2$ and $T = (V_2, E_2)$  be a tree with  $ V_2 = \{ v_1, v_2, \ldots, v_{n} \}$, where $n \ge 2$. Moreover, let $V_1 \cap V_2 = \{ v \}$.
Then the 1-fusion $G_1 \oplus T$ is an edge intersection hypergraph of a 3-uniform hypergraph.
\end{theorem}

\begin{proof}
If $T$ is an edge intersection hypergraph of a 3-uniform hypergraph $\cH_2=(V_2, \cE_2)$, owing to Theorem 3 there is nothing to show. So we have to consider only the exceptional trees $P_2, P_3, P_4, P_5, P_6, T7$ and $T12$ from Theorem 1. Note that the one and only case where the condition $d_{G_1}(v) \ge 2$ will be needed is the path $P_2$. At first we investigate the paths $P_n =(V_2, E_2)$ with $E_2 = \{ \{v_1, v_2 \}, \ldots, \{v_{n-1}, v_n \} \}.$

In advance, we mention that $|V_1| \ge 4$ is valid, since $G_1$ is an edge intersection hypergraph of a 3-uniform hypergraph contaning the non-isolated vertex $v \in V_1$.

\begin{enumerate}
\item[(a)] {\bf $T=P_2$.}\\[0.5ex]
Let $v = v_1$ and $u, w \in V_1$ be two neighbors of $v_1$ in $G_1$. Then  $G_1 \oplus P_2 = EI(\cH)$ with
$\cH = ( V_1 \cup \{v_2 \}, \cE_1 \cup \{ \{u, v_1, v_2 \}, \{w, v_1, v_2 \} \})$.

\item[(b)] {\bf $T \in \{P_3, P_4, P_5, P_6 \}$.}\\[0.5ex]
For $3 \le  n \le 6$, let   $i \le n$, $v = v_i$ and $u, w \in V_1$, where $u$ is a neighbor of $v_i$ in $G_1$.

\smallskip
First we consider the  situation $n=3$ and  $v= v_2$, i.e. $v$ is the inner vertex of $P_3$. Then the 1-fusion $G_1 \oplus P_3$ is nothing else than the 2-fusion $G_1 \oplus K_{1,3}$, where $K_{1,3} = (V_2 \cup \{ u \}, E_2 \cup \{ \{ v_2, u \} \} )$. Theorem 3 implies that  this $2$-fusion of $G_1$ and $K_{1,3}$ has the required properties.

\smallskip
Next we
investigate the $1$-fusion $G_1 \oplus P_n$ in the end vertex $v = v_1$ of $P_n$.\\
For this end, we consider the hypergraph
$\cH = ( V_1 \cup \{v_2, \ldots, v_{n} \}, \cE_1 \cup \{ \{u, v_1, v_2 \},
\{ v_1, v_2, v_3 \}, $ $\{ v_2, v_3, v_4 \}, \ldots, \{ v_{n-2}, v_{n-1}, v_{n} \},  \{w, v_{n-1}, v_{n} \} \})$ and obtain $G_1 \oplus P_n = EI(\cH)$.

The remaining cases ($n \ge 4$ and  the vertex  $v_i \in V_1 \cap V_2$ is an inner vertex  of the path $P_n$, i.e. $1 < i < n$) can be obtained from the results above in two steps.\\
Let $P'_n = (V'_2= \{ v_1, v_2, \ldots, v_{i} \}, E'_2 = \{ \{ v_1, v_2 \}, \{ v_2, v_3 \}, \ldots, \{v_{i-1}, v_{i} \} \} )$ and $P''_n = (V''_2= \{ v_i, v_{i+1}, \ldots, v_{n} \}, E'_2 = \{ \{ v_i, v_{i+1} \}, \{ v_{i+1}, v_{i+2} \}, \ldots, \{v_{n-1}, v_{n} \} \} )$. Assume, $|V'_2| \ge |V''_2|$. Then $P'_n$ contains at least three vertices and the assumption of case (b) is fulfilled. In a first step,
$G_1 \oplus P'_n$ and, in a second step, $(G_1 \oplus P'_n)\oplus P''_n$ is an edge intersection hypergraph of a 3-uniform hypergraph, respectively. Only in the second step, when $n-i = 1$ holds (i.e. $P''_n$ contains exactly two vertices) we have to make use of part (a). In this case, the vertex $v_i$ has minimum degree 2 in $G_1 \oplus P'_n$; hence part (a) is applicable.

\end{enumerate}

\noindent
Now we come to the 1-fusion of $G_1$ and the exceptional trees $T7$ and $T12 $.

\begin{enumerate}
\item[(c)] {\bf $T \in \{ T7, T12 \}$.}\\[0.5ex]
Trivially, every vertex in $V_2 = V(T)$ is included in a path of length 4 in $T$.  So let $P_4 $ be such a path in $T$ containing the vertex $v \in V_1 \cap V_2$.
Obviously, $T$ is a 1-fusion of $P_4$ and a second path $P_t$, such that $T = P_4 \oplus P_t$ and $V(P_4) \cap V(P_t) = \{ v'\}$ for a certain vertex $v' \in V_2$ with $d_T(v') = 3$. Clearly, for $T = T7$ we have $t=2$ and for $T = T12$ we get $t=3$.

Part (b) provides that the 1-fusion $G_1 \oplus P_4$ is an edge intersection hypergraph of a 3-uniform hypergraph.
In order to obtain the final 1-fusion $G_1 \oplus T$ from $G_1 \oplus P_4$ it suffices to add the path $P_t$, i.e. $G_1 \oplus T = (G_1 \oplus P_4) \oplus P_t$. In dependence on $t$, part (a)  and part (b), respectively, provides that $G_1 \oplus T$ is an edge intersection hypergraph of a 3-uniform hypergraph.

Note that the assumption for (a) is fulfilled, since $d_{G_1 \oplus P_4}(v') \ge d_T(v') - 1 = 2$, i.e. the vertex $v'$ has at least two neighbours in   $G_1 \oplus P_4$.\hqed

\end{enumerate}

\end{proof}


\begin{remark}
Because in the above proof the condition $d_{G_1}(v) \ge 2$ is needed only for $T = P_2$, this condition in Theorem 4 can be weakened to $d_{G_1}(v) \ge 1$ if we restrict ourselves on trees with at least 3 vertices.
\end{remark}

\begin{theorem}
Let the graph $G_1 = (V_1, E_1)$ be the edge intersection hypergraph of a 3-uniform hypergraph $\cH_1=(V_1, \cE_1)$, $v \in V_1$ with $d_{G_1}(v) \ge 1$ and $C_n = (V_2, E_2)$  be the cycle of length $n \ge 3$. Moreover, let $V_1 \cap V_2 = \{ v \}$ and, for $n=4$, the number of vertices in $G_1$ be at least 5.
Then the 1-fusion $G_1 \oplus C_n$ is an edge intersection hypergraph of a 3-uniform hypergraph.
\end{theorem}

\begin{proof}
For $n\ge 5$, Theorem 2 provides that $C_n =( V_2, E_2)$ is an edge intersection hypergraph of a 3-uniform hypergraph $\cH_2=(V_2, \cE_2)$. Therefore, owing to Theorem 3 there is nothing to show in this case.

So we have to investigate only the cycles $C_3$ and $C_4$.
For this end let $V_2 = \{ v_1, v_2, \ldots, v_n \}$, $E_2 = \{ \{ v_1, v_2 \}, \ldots, \{ v_{n-1}, v_n \}, \{v_n, v_1 \} \}$ and $v_1 = v \in V_1 \cap V_2$. Moreover, let $u \in V_1$ be a neighbor of $v_1$ in the graph $G_1$. As mentioned at the beginning of the proof of Theorem 4, we have $|V_1| \ge 4$ and so we can choose a vertex $x \in V_1 \setminus \{ v_1, u \}$.

\begin{enumerate}
\item[(a)] {\bf $n=3$.}\\[0.5ex]
We consider $\cH = ( V_1 \cup \{v_2, v_3 \}, \cE_1 \cup \{ \{ v_1, v_2, v_3 \}, \{u, v_1, v_2 \}, \{u, v_1, v_3 \}, \{x, v_2, v_3 \} \})$.

In comparison with $G_1$, the new hyperedges $\{ v_1, v_2, v_3 \}, \{u, v_1, v_2 \}, \{u, v_1, v_3 \},$ $ \{x, v_2, v_3 \}$ induce no additional edges in $(G_1 \oplus C_3) \setminus \{v_2, v_3 \}$: \\
Clearly, $ \{u, v_1, v_2 \} \cap \{u, v_1, v_3 \} = \{ u, v_1 \}$ is always an edge in $G_1$, namely $\{ u, v_1 \} = \{ u, v \} \in E_1$. Moreover, the hyperedge $ \{ v_1, v_2, v_3 \}$ and $ \{x, v_2, v_3 \}$ has only one vertex with $G_1$ in common, namely the vertex $v_1$ and $x$, respectively.

Hence,  we obtain  $G_1 \oplus C_3 = EI(\cH)$.

\item[(b)] {\bf $n=4$.}\\[0.5ex]
Now we have $|V_1| \ge 5$ and there are two additional vertices $y, z \in V_1 \setminus \{ v_1, u, x \}$ such that $ v_1, u, x, y, z$ are pairwise distinct. \\
We use the hypergraph $\cH = ( V_1 \cup \{v_2, v_3, v_4 \}, \cE(\cH))$ with \\
$ \cE(\cH)) = \cE_1 \cup \{ \{ v_1, v_2, v_3 \},  \{ v_1, v_2, v_4 \},  \{u, v_1, v_4 \}, \{x, v_2, v_3 \}, \{ y, v_3, v_4 \}, \{ z, v_3, v_4 \} \}$.

Analogously to the previous case, it can be verified that $G_1 \oplus C_4 = EI(\cH)$ holds. \hqed
\end{enumerate}

\end{proof}

\section{Cacti as edge intersection hypergraphs}

The basic idea is to decompose a given cactus $G$  into cycles and (in a certain sense maximal) trees.
After that we begin with a cycle $C_n$ of length $n \ge 5$ and a tree $T \notin \{P_1, P_2, \ldots, P_6, T7, T12 \}$, respectively, (which is an edge intersection hypergraph of a 3-uniform hypergraph) and reconstruct the original cactus step by step using Theorem 4 and Theorem 5.
For this purpose, in each step we build a 1-fusion of a (connected) subgraph $G_i$ (which is an edge intersection hypergraph of a 3-uniform hypergraph $\cH_i$) of the original cactus and one of the cycles or trees described above.

We begin with the decomposition. First of all, let $G=(V,E)$ be a cactus and $V'= \{ v_1, \ldots, v_s \} \subseteq V$ be the set of all vertices having a degree of at least 3 and being contained in a cycle of $G$. Evidently, the vertices in $V'$ are articulation vertices of $G$.
We refer to these vertices  as the {\em decomposition vertices} or shortly {\em d-vertices} of $G$. It is easy to see that $V' = \emptyset$ is equivalent to the case that the cactus $G$ is a cycle or a tree, respectively. In this case Theorem 1 and Theorem 2 include the corresponding characterizations.
So in the following assume $V' \neq \emptyset$.

\smallskip
The so-called {\em tree-cycle-decomposition} of the cactus $G$ will be carried out in two consecutive steps.

\smallskip
\noindent
\ul{Step 1.} In each cycle of $G$, we delete all  edges and all vertices of degree 2. Thus we obtain a forest consisting of pairwise vertex-disjoint trees $T^1, T^2, \ldots, T^k$, the so-called {\em limbs} of $G$.
Clearly, every limb contains at least one $d$-vertex. If the tree $T^j$ is a single vertex, i.e. $V(T^j) = \{ v_j \}$, then $v_j$ is an articulation vertex that connects $\frac{1}{2} d_G(v_j)$ cycles in $G$.
Note that the limbs are the  "(in a certain sense maximal) trees" mentioned at the beginning of the section.

\smallskip
\noindent
\ul{Step 2.}
We start again with the original cactus $G$ and delete all edges $e \in E(T^1), \ldots, E(T^k)$. Besides several isolated vertices, this leads to a system of Eulerian graphs which can be uniquely decomposed into a system $C^1, \ldots, C^l$ of pairwise edge-disjoint cycles. We denote $C^1, \ldots, C^l$ as {\em the cycles} of $G$.

\medskip
The next remark is a collection of some simple, but useful properties of the tree-cycle-decomposition.
\begin{remark}
\begin{enumerate}
\item[(i)] The $d$-vertices $v_1, \ldots, v_s$, the limbs $T^1, \ldots, T^k$ as well as the cycles $C^1, \ldots, C^l$ are uniquely determined.
\item[(ii)] The limbs $T^1, \ldots, T^k$ are pairwise vertex-disjoint.
\item[(iii)] Any $d$-vertex  is contained in at least one of the cycles of $G$.
\item[(iv)] Any $d$-vertex  is contained in at most one of the limbs of $G$.
\item[(v)]  If $V' \neq \emptyset$, then every cycle  and every limb  includes at least one $d$-vertex.
\end{enumerate}
\end{remark}

\smallskip
Now we are ready to formulate our main theorem.

\smallskip
\begin{theorem}
Let $G = ( V, E)$ be a cactus with
\vspace{-1mm}
\begin{enumerate}
\item[(a)] circumference $ci(G) \ge 5$ \quad or
\vspace{-2mm}
\item[(b)] $G$ contains a limb $T \notin \{P_1, P_2, \ldots, P_6, T7, T12 \}$.
\end{enumerate}
\vspace{-1mm}
Then $G$ is  an edge intersection hypergraph of a 3-uniform hypergraph.
\end{theorem}

\begin{proof}
Let $ \{ G_1, G_2, \ldots, G_{k+l} \} = \{ C^1, C^2, \ldots, C^l, T^1, T^2, \ldots, T^k \}$, where $C^1, C^2, \ldots, C^l$ and \linebreak
$T^1, T^2, \ldots, T^k$ are the cycles and the limbs of $G$, respectively.

\medskip
\noindent
\ul{Case (a): $ci(G) \ge 5$.}

\smallskip
Let the indices of the subgraphs $G_1, \ldots, G_{k+l}$ be chosen so that $G_1 = C^1$ is a cycle of a length of at least 5 and, for every $p \in \{ 1, 2, \ldots, k+l-1 \}$, the subgraph $G_1 \cup \ldots \cup G_p$ of the cactus $G$ has a vertex $v_p$ in common with the subgraph $G_{p+1}$. Obviously,
$\big(V(G_1) \cup \ldots \cup V(G_p)\big) \cap V(G_{p+1}) = \{ v_p \} \subseteq V'$, where $V'$ is the set of the $d$-vertices of $G$.

Trivially, $G = G_1 \cup \ldots \cup G_{k+l}$ and, for any $p \in \{ 1, 2, \ldots, k+l-1 \}$, from $|V(G_{p+1})| = 1$ (this is the case if and only if $G_{p+1}$ is a trivial tree containing only one vertex, i.e. $V(G_{p+1}) = \{ v_p \}$) it follows $(G_1 \cup \ldots \cup G_p) \cup G_{p+1} =  G_1 \cup \ldots \cup G_p$. Owing to our indexing of $G_1, G_2, \ldots, G_{k+l}$ this is equivalent to
$( \ldots((G_1 \oplus G_2) \oplus G_3) \ldots \oplus G_p) \oplus G_{p+1} =  ( \ldots((G_1 \oplus G_2) \oplus G_3) \ldots \oplus G_p)$. Note that all the clique-fusions in this expression are 1-fusions and, additionally, all these clique-fusions are connected.

To mention the most trivial case, the 1-fusion of any graph $G' = ( V(G'), E(G'))$ with a graph $G'' = ( V(G'') = \{ v' \}, E(G'') = \emptyset)$, with $v' \in V(G')$, is the original graph $G'$ itself.
Therefore, if $G'$ is an edge intersection hypergraph of a 3-uniform hypergraph $\cH$, then also $G' \oplus G'' = EI(\cH)$ is valid.

According to the assumption, $G_1$ is an edge intersection hypergraph of a 3-uniform hypergraph $\cH_1$ containing at least 5 vertices. Therefore $G_1$ fulfills also the assumptions of Theorems 4 and 5 and $G_1 \oplus G_2$ is also an edge intersection hypergraph of a 3-uniform hypergraph $\cH_2$. Since $G_1$ is a cycle, all of its vertices have degree $2$ and it plays no role whether or not $G_2$ is a cycle or a limb.

With two little additional arguments we can argue in  the same manner for arbitrarily chosen $p \in \{  2, \ldots, k+l-1 \}$ considering $( \ldots((G_1 \oplus G_2) \oplus G_3) \ldots \oplus G_p) \oplus G_{p+1}$.

First, if  $( \ldots((G_1 \oplus G_2) \oplus G_3) \ldots \oplus G_p)$ is an edge intersection hypergraph of a 3-uniform hypergraph $\cH_p$ then it contains at least 5 vertices (since $|V(G_1)| \ge 5$ holds). This implies that the clique-fusion $( \ldots((G_1 \oplus G_2) \oplus G_3) \ldots \oplus G_p) \oplus G_{p+1}$ also in the case $G_{p+1} = C_4$ is an edge intersection hypergraph of a 3-uniform hypergraph. Clearly, also for $G_{p+1} = C_3$ no problem occurs.

Secondly, in case of $G_{p+1} = P_2$, we have to ensure that the vertex \\
$v_p \in \big(V(G_1) \cup \ldots \cup V(G_p)\big) \cap V(G_{p+1})$ has a degree $d_{G_1 \cup \ldots \cup G_p}(v_p) \ge 2$ (cf. Theorem 4 and Remark 2). This yields from the definition of the $d$-vertices and the limbs as well as from the construction of our tree-cycle-decomposition of the cactus $G$ in the following way.
In case of $G_{p+1} = P_2$ the graph $G_{p+1}$ is a limb. Remark 3(iv) includes that the $d$-vertex $v_p \in V(G_{p+1}$ cannot be contained in another limb of $G$. Therefore $v_p$ is included in a cycle of $(\ldots((G_1 \oplus G_2) \oplus G_3) \ldots  \oplus G_p) = G_1 \cup \ldots \cup G_p $ and has a degree $d_{G_1 \cup \ldots \cup G_p}(v_p) \ge 2$.

This completes the proof of Case (a).

\medskip
\noindent
\ul{Case (b): $ci(G) \le 4$.}

\smallskip
We use nearly the same argumentation as in Case (a), the only modification is that  we have to start with $G_1 =T^1$ instead of $G_1 = C^1$, where $T^1 \notin  \{P_1, P_2, \ldots, P_6, T7, T12 \}$ is a limb of $G$.
If $G=T^1$ holds then there is nothing to show.

So assume $k+l > 1$ and choose the indices of $G_2, G_3, \ldots, G_{k+l}$ in the same way as in Case (a).
Thus the graph $G_1=T^1$ is again an edge intersection hypergraph of a 3-uniform hypergraph. Moreover, $G_2 = C_3$ or $G_2 = C_4$ and it suffices to prove that any fusion $G_1 \oplus G_2 = T^1 \oplus C_n$ ($n \in \{ 3, 4 \}$) is an edge intersection hypergraph of a 3-uniform hypergraph. The rest of the argumentation can be taking over word-for-word from Case (a). So let us consider $T^1 \oplus C_n$.

Since $T^1$ is an edge intersection hypergraph of a 3-uniform hypergraph, we can apply Theorem 5. The only exception is the case $|V(T^1)| = 4$ and $n=4$. Obviously, this  corresponds to $T^1 = K_{1,3}$ and we have to investigate the two possible 1-fusions of $K_{1,3}$ and $C_4$.

For this end, let  $K_{1,3} =( V_1 = \{ v_1, v_2, v_3, v_4 \}, E_1)$, $C_4 =( V_2 = \{v_4, v_5, v_6, v_7 \}, E_2)$, where  $ E_2 = \{ \{ v_4, v_5 \}, \{ v_5, v_6 \}, \{ v_6, v_7 \}, \{ v_7, v_4 \} \})$, and look at $ K_{1,3} \oplus C_4$. In the following, we discuss both 1-fusions.

\bigskip
The first 1-fusion we have to investigate is the situation  that $v_4$ is not the center of the star; so w.l.o.g. let $v_1$ be the center, i.e. $E_1 = \{ \{ v_1, v_2 \}, \{ v_1, v_3 \},\{ v_1, v_4 \} \}$. Then the hypergraph \\
$\cH = (V, \cE)$ with
$\cE = \{ \{ v_1, v_2, v_3 \},
\{ v_1, v_2, v_4 \},
\{ v_1, v_3, v_4 \},
\{v_1, v_4, v_7 \},
\{ v_1, v_5, v_6 \},
\{ v_2, v_6, v_7 \}, $ \\
$
\{ v_3, v_6, v_7 \},
\{ v_4, v_5, v_6 \},
\{v_4, v_5, v_7 \} \}$
provides the 1-fusion $EI(\cH) = K_{1,3} \oplus C_4 = ( V, E )$ with $V = V_1 \cup V_2$ and
$E = \{ \{ v_1, v_2 \}, \{ v_1, v_3 \},\{ v_1, v_4 \}, \{ v_4, v_5 \}, \{ v_5, v_6 \}, \{ v_6, v_7 \}, \{ v_7, v_4 \} \}$ (see the first picture in Fig. 3).

\bigskip
Now let $v_4$ be the center of the star. Thus we have  $E_1 = \{ \{ v_1, v_4 \}, \{ v_2, v_4 \},\{ v_3, v_4 \} \}$.
We consider the hypergraph
$\cH = (V, \cE)$ with
$\cE = \{
\{ v_1, v_3, v_4 \},
\{ v_1, v_4, v_7 \},
\{v_2, v_3, v_4 \},
\{ v_2, v_4, v_5 \}, $ \\ $
\{ v_2, v_4, v_7 \},
\{ v_3, v_6, v_7 \},
\{ v_4, v_5, v_6 \},
\{v_5, v_6, v_7 \} \}$. $\cH$ has the edge intersection hypergraph  $EI(\cH) = K_{1,3} \oplus C_4 = ( V, E )$ with $V = V_1 \cup V_2$ and
$E = \{ \{ v_1, v_4 \}, \{ v_2, v_4 \},\{ v_3, v_4 \}, \{ v_4, v_5 \}, \{ v_5, v_6 \}, \{ v_6, v_7 \}, $ $ \{ v_7, v_4 \} \}.$ This is the second possible 1-fusion of  $K_{1,3}$ and $C_4$  (see the second picture in Fig. 3).

\hqed
\end{proof}

\vspace{-6mm}
\providecolor{black}{rgb}{0,0,0}
	\begin{figure}[h]
	\centering
	\begin{tikzpicture}


		\path[solid,line cap=round,draw=black, line width=0.030499075785582256cm] (1.8299445471349354,4.015711645101664) --  (0.6709796672828097,5.17467652495379);

		\path[solid,line cap=round,draw=black, line width=0.030499075785582256cm] (1.8299445471349354,4.015711645101664) --  (2.9889094269870613,5.17467652495379);

		\path[solid,line cap=round,draw=black, line width=0.030499075785582256cm] (1.8299445471349354,4.015711645101664) --  (1.8299445471349354,2.856746765249538);

		\path[solid,line cap=round,draw=black, line width=0.030499075785582256cm] (1.8299445471349354,0.5388170055452866) --  (0.6709796672828097,1.6977818853974123);

		\path[solid,line cap=round,draw=black, line width=0.030499075785582256cm] (1.8299445471349354,2.856746765249538) --  (0.6709796672828097,1.6977818853974123);

		\path[solid,line cap=round,draw=black, line width=0.030499075785582256cm] (9.170055452865064,2.856746765249538) --  (8.01109057301294,4.015711645101664);

		\path[solid,line cap=round,draw=black, line width=0.030499075785582256cm] (9.170055452865064,2.856746765249538) --  (9.170055452865064,4.015711645101664);

		\path[solid,line cap=round,draw=black, line width=0.030499075785582256cm] (9.170055452865064,2.856746765249538) --  (10.32902033271719,4.015711645101664);

		\path[solid,line cap=round,draw=black, line width=0.030499075785582256cm] (9.170055452865064,2.856746765249538) --  (8.01109057301294,1.6977818853974123);

		\path[solid,line cap=round,draw=black, line width=0.030499075785582256cm] (9.170055452865064,2.856746765249538) --  (10.32902033271719,1.6977818853974123);

		\path[solid,line cap=round,draw=black, line width=0.030499075785582256cm] (9.170055452865064,0.5388170055452866) --  (8.01109057301294,1.6977818853974123);

		\path[solid,line cap=round,draw=black, line width=0.030499075785582256cm] (9.170055452865064,0.5388170055452866) --  (10.32902033271719,1.6977818853974123);

		\path[solid,line cap=round,draw=black, line width=0.030499075785582256cm] (2.9889094269870613,1.6977818853974123) --  (1.8299445471349354,2.856746765249538);

		\path[solid,line cap=round,draw=black, line width=0.030499075785582256cm] (2.9889094269870613,1.6977818853974123) --  (1.8299445471349354,0.5388170055452866);

		\fill[fill=black] (1.8299445471349354,4.015711645101664) circle (0.07624768946395565);
		\pgftext[x=2.236598890942699cm,y=4.015711645101664cm]{$v_{1}$}

		\fill[fill=black] (0.6709796672828097,5.17467652495379) circle (0.07624768946395565);
		\pgftext[x=0.6709796672828097cm,y=5.581330868761553cm]{$v_{2}$}

		\fill[fill=black] (2.9889094269870613,5.17467652495379) circle (0.07624768946395565);
		\pgftext[x=2.9889094269870613cm,y=5.581330868761553cm]{$v_{3}$}

		\fill[fill=black] (1.8299445471349354,2.856746765249538) circle (0.07624768946395565);
		\pgftext[x=2.236598890942699cm,y=2.856746765249538cm]{$v_{4}$}

		\fill[fill=black] (0.6709796672828097,1.6977818853974123) circle (0.07624768946395565);
		\pgftext[x=0.2643253234750462cm,y=1.6977818853974123cm]{$v_{7}$}

		\fill[fill=black] (1.8299445471349354,0.5388170055452866) circle (0.07624768946395565);
		\pgftext[x=1.8299445471349354cm,y=0.1321626617375231cm]{$v_{6}$}

		\fill[fill=black] (9.170055452865064,2.856746765249538) circle (0.07624768946395565);
		\pgftext[x=9.576709796672828cm,y=2.856746765249538cm]{$v_{4}$}

		\fill[fill=black] (8.01109057301294,1.6977818853974123) circle (0.07624768946395565);
		\pgftext[x=7.604436229205176cm,y=1.6977818853974123cm]{$v_{7}$}

		\fill[fill=black] (9.170055452865064,0.5388170055452866) circle (0.07624768946395565);
		\pgftext[x=9.170055452865064cm,y=0.1321626617375231cm]{$v_{6}$}

		\fill[fill=black] (10.32902033271719,1.6977818853974123) circle (0.07624768946395565);
		\pgftext[x=10.735674676524955cm,y=1.6977818853974123cm]{$v_{5}$}

		\fill[fill=black] (8.01109057301294,4.015711645101664) circle (0.07624768946395565);
		\pgftext[x=8.01109057301294cm,y=4.422365988909427cm]{$v_{1}$}

		\fill[fill=black] (9.170055452865064,4.015711645101664) circle (0.07624768946395565);
		\pgftext[x=9.170055452865064cm,y=4.422365988909427cm]{$v_{2}$}

		\fill[fill=black] (10.32902033271719,4.015711645101664) circle (0.07624768946395565);
		\pgftext[x=10.32902033271719cm,y=4.422365988909427cm]{$v_{3}$}

		\fill[fill=black] (2.9889094269870613,1.6977818853974123) circle (0.07624768946395565);
		\pgftext[x=3.3955637707948245cm,y=1.6977818853974123cm]{$v_{5}$}
	\end{tikzpicture}

		\caption{The two 1-fusions of $K_{1,3}$ and $C_4$}
	\end{figure}
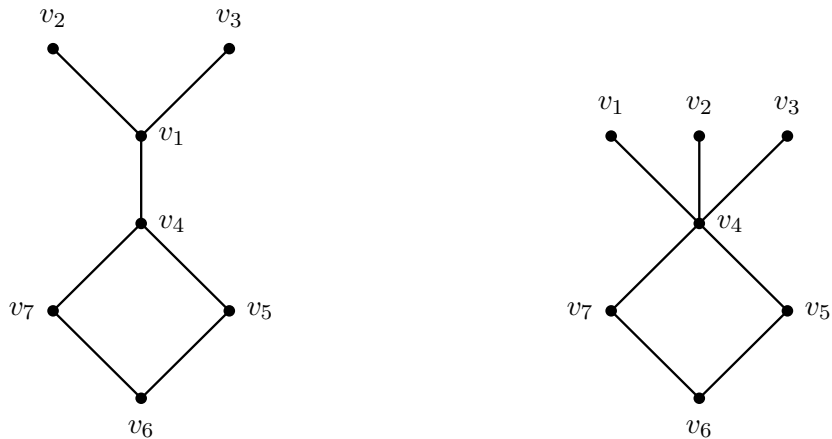

\section{Concluding remarks}

In Section 3 we made use of  very special clique-fusions, namely the (iterated) 1-fusion of edge intersection hypergraphs of 3-uniform hypergraphs. Since cacti can be decomposed into limbs and cycles using the $d$-vertices, the 1-fusion had been proved to be a suitable tool  for our investigations on cacti.

Besides cycles and trees also other classes of graphs are known to be edge intersection hypergraphs of 3-uniform hypergraphs, e.g. as wheels and  complete graphs with at least 5 and 4 vertices, respectively. So it would be interesting to ask for more  classes of graphs being edge intersection hypergraphs which can be constructed from known such graphs using the clique-fusion. Besides the 1-fusion also the 2-fusion seems to be a relatively simple tool for such constructions.

Remember that in Section 3 we restricted ourselves on the {\em iterated} clique-fusion in the sense, that in each step we applied the 1-fusion only to two graphs (see the proof of Theorem 6 where we dealt with the clique fusion in the form $( \ldots((G_1 \oplus G_2) \oplus G_3) \ldots \oplus G_p) \oplus G_{p+1}$. Another interesting topic for further investigations could be the usage of the $k$-fusion in a more common sense, namely as  $G_1 \oplus \ldots \oplus G_r$, where $r > 2$ holds (see the original definition at the beginning of Section 2). We conjecture that -- using the clique-fusion in this common sense -- the construction of corresponding classes of edge intersection hypergraphs of 3-uniform hypergraphs may be much more complicated.

Finally, we formulate three problems.

\begin{problem}
Using the clique-fusion, find more classes of edge intersection hypergraphs of 3-uniform hypergraphs.
\end{problem}

\begin{problem}
Besides the clique-fusion, find other tools for the construction  of graphs being  edge intersection hypergraphs of $r$-uniform $(r \ge 3)$ or non-uniform hypergraphs.
\end{problem}

\begin{problem}
Find more classes of graphs being edge intersection hypergraphs of $r$-uniform $(r \ge 3)$ or non-uniform hypergraphs.
\end{problem}

Instead of graphs, also the investigation of hypergraphs which are edge intersection hypergraphs of certain hypergraphs is an attractive topic.

%
%
%
%
%
%
%
%

\end{document}